\documentclass{article}
\usepackage{amssymb}
\usepackage{amsfonts}
\usepackage{amsmath}

\newtheorem{theorem}{Theorem}

\newtheorem{corollary}[theorem]{Corollary}

\newtheorem{lemma}[theorem]{Lemma}

\newtheorem{proposition}[theorem]{Proposition}
\newtheorem{remark}[theorem]{Remark}

\newenvironment{proof}[1][Proof]{\noindent\textbf{#1.} }{\ \rule{0.5em}{0.5em}}
\begin{document}

\title{Self-repelling diffusions via an infinite dimensional approach}
\author{Michel Bena\"{\i}m, Ioana Ciotir, Carl-Erik Gauthier\thanks{We acknowledge financial support from the Swiss National Foundation Grant FN
$200020\_ 149871/1$.} }
\maketitle

\begin{abstract}
In the present work we study self-interacting diffusions following an
infinite dimensional approach. First we prove existence and uniqueness of a
solution with Markov property. Then we study the corresponding transition
semigroup and, more precisely, we prove that it has Feller property and we
give an explicit form of an invariant probability of the system.
\end{abstract}


MSC 60K35, 60H10, 60H30

\section{Introduction}
\label{intro}
In the present work we are interested in stochastic differential equations
of the type 
\begin{equation}
X_{t}=x+\int_{0}^{t}g\left( X_{s}\right)
ds-\int_{0}^{t}\int_{0}^{s}f^{\prime }\left( X_{s}-X_{r}\right) drds+\beta
_{t}  \label{ecu}
\end{equation}%
where $x$$\in $$\mathbb{R}$, $\beta _{t}$ is a standard 1D Brownian motion
and $f$ is a $2\pi $- periodic function with sufficient regularity. The
initial drift profile $g$ shall be chosen in a convenient form detailed
below, in order to assure the Markov property of the process.

The motivating example of this equation comes from physics, and more
precisely from systems that model the shape of a growing polymer.

A first model was introduced in the framework of random walks by Coppersmith
and Diaconis in \cite{cop.diac} and intensively studied later (see \cite{benaim}, \cite{davis}, \cite{permantle} ). The continuous time
corresponding processes were also studied under different assumptions on $f$.

One of the first papers was published by Norris, Rogers and Williams in 1987
and gives a Brownian model with local time drift for self-avoiding random
walk, i.e., 
\begin{equation*}
X_{t}=\beta _{t}-\int_{0}^{t}g\left( X_{s},L\left( s,X_{s}\right) \right) ds
\end{equation*}
where $\left\{ L\left( t,x\right) ;\,t\geqslant 0,\,x\in \mathbb{R}\right\} $
is the local time process of $X$. The main difficulty in this approach is
the lack of Markov property (see \cite{NRW}).

In 1992 Durrett and Rogers studied asymptotic behavior of Brownian polymers.
More precisely they are interested in processes of the form 
\begin{equation*}
X_{t}=\beta _{t}+\int_{0}^{t}\int_{0}^{s}f\left( X_{s}-X_{r}\right) drds
\end{equation*}%
where $f(x)=\Psi \left( x\right) x/\left\Vert x\right\Vert, \Psi(x)\geqslant 0$ (see \cite{durrett}).

An extended study was also made by Bena\"{\i}m, Ledoux, Raimond in the
series of papers on self interacting diffusions (see \cite{benaim1}, \cite{benaim2}, \cite{benaim3}).

In a recent paper, Tarr\`es, T\'oth and Valk\'o proved that a smeared-out version
of the local time function from the point of view of the actual position of
the process is Markov (see \cite{TTV})

In the present work we study equation (\ref{ecu}) following an infinite
dimensional approach. In fact we show that, by choosing a particular form
for the initial drift profile $g$ and by taking the Fourier development of
the function $f$, the stochastic differential equation becomes equivalent to
a system in $\mathbb{R}\times l^{2}\times l^{2}$. Consequently, the problem
can be treated by using tools from the theory of stochastic differential
equations in infinite dimensions and we show existence and uniqueness of the
solution with Markov property.

Then we prove Feller property for the transition semigroup and we show that
the system has an invariant probability measure which is explicitly given.

In the sequel, we denote by $C([0,\infty];H)$ the space of continuous functions from $[0,\infty]$ to the Hilbert space $H$, by $C_{b}^{k}(H)$ the space of bounded functions from $H$ to $\mathbb{R}$ that are $k$ times continuously Fr\'echet differentiable with bounded derivatives up to order $k$, and by $L_{loc}^{\infty}(0,\infty;H)$ the space of functions from $(0,\infty)$ to $H$ which are locally $L^{\infty}.$

\section{Equivalence with an infinite dimensional system}

Consider the stochastic differential equation 
\begin{equation}
X_{t}=x+\int_{0}^{t}g\left(X_{s}\right)ds-\int_{0}^{t}\int_{0}^{s}f^{\prime
}\left(X_{s}-X_{r}\right)drds+\beta_{t}  \label{ecu.init}
\end{equation}
for $x\in\mathbb{R}$ and $\beta_{t}$ a standard 1D Brownian motion.

We assume that $f$ is an even, $2\pi $ periodical function and sufficiently
regular such that the coefficients $\left( a_{n}\right) _{n}$ of the
corresponding Fourier series 
\begin{equation}
f\left( x\right) \thicksim \frac{a_{0}}{2}+\sum_{n=1}^{\infty }a_{n}\cos
\left( nx\right)  \label{fourier}
\end{equation}
form a positive rapidly decreasing sequence and $a_{n}>0,$ for all $n\in 
\mathbb{N}$. For reader's convenience, we recall the definition of the space
of rapidly decreasing sequences of order $k$
\begin{equation}
O^{k}=\left\{ \left( a_{n}\right) _{n};\,\sum_{n=1}^{\infty }\left(
1+n^{2}\right) ^{k}a_{n}^{2}<\infty \right\} .  \label{space}
\end{equation}
In our case $\left( a_{n}\right) _{n}$ is assumed to belong at least to $
O^{5}$ and for that it is sufficient to have $f$ in the Sobolev space $H_{2\pi }^{5}\left( 
\mathbb{R}\right) $ of $2\pi$ periodic functions.

We choose an initial drift profile $g$ of the form 
\begin{equation}
g\left( x\right) =\sum_{n}a_{n}^{1/2}n\left( u_{0}^{n}\sin \left( nx\right)
+v_{0}^{n}\cos \left( nx\right) \right) ,  \label{drift}
\end{equation}
where $\left( u_{0}^{n}\right) _{n}$ and $\left( v_{0}^{n}\right) _{n}$ are
two arbitrary sequences from $l^{2}$.

Since $f^{\prime}$ and $g$ are both $2\pi$-periodic, $(X_{t})_{t\geqslant 0}$ might be interpreted as an angle. Consequently $X_{t}$ could be identified to the point $(\cos (X_{t}),\sin(X_{t}))\in \mathbb{S}^{1}$. For more details see for example \cite{DKMS12}.

By standard computation we see that 
\begin{eqnarray*}
-f^{\prime }\left( X_{s}-X_{r}\right) &=&\sum_{n}a_{n}^{1/2}n\,\sin \left(
nX_{s}\right) \left( a_{n}^{1/2}\cos \left( nX_{r}\right) \right) \\
&&-\sum_{n}a_{n}^{1/2}n\,\cos \left( nX_{s}\right) \left( a_{n}^{1/2}\sin
\left( nX_{r}\right) \right).
\end{eqnarray*}

If we replace\emph{\ }(\ref{fourier}) and (\ref{drift}) in (\ref{ecu.init})
and set 
\begin{equation*}
u_{t}^{n}=u_{0}^{n}+a_{n}^{1/2}\int_{0}^{t}\cos \left( nX_{s}\right) ds
\end{equation*}
\begin{equation*}
v_{t}^{n}=v_{0}^{n}-a_{n}^{1/2}\int_{0}^{t}\sin \left( nX_{s}\right) ds
\end{equation*}
we can rewrite equation (\ref{ecu.init}) as a system in the Hilbert space $H=\mathbb{R}\times l^{2}\times l^{2}$ as
\begin{equation*}
\left\{ 
\begin{array}{l}
X_{t}=x+\int_{0}^{t}\sum_{n}n\left( a_{n}^{1/2}\sin \left( nX_{s}\right)
u_{s}^{n}+a_{n}^{1/2}\cos \left( nX_{s}\right) v_{s}^{n}\right) ds+\beta
_{t},\medskip \\ 
u_{t}^{n}=u_{0}^{n}+a_{n}^{1/2}\int_{0}^{t}\cos \left( nX_{s}\right)
ds,\quad n\geq 1,\medskip \\ 
v_{t}^{n}=v_{0}^{n}-a_{n}^{1/2}\int_{0}^{t}\sin \left( nX_{s}\right)
ds,\quad n\geq 1,
\end{array}
\right.
\end{equation*}
or equivalently as a stochastic differential equation in a Hilbert space 
\begin{equation*}
Y_{t}=y+\int_{0}^{t}F\left( Y_{s}\right) ds+\sigma dW_{t}
\end{equation*}
where the process
\begin{equation*}
Y_{t}=\left( X_{t},~\left( u_{t}^{n}\right) _{n},~\left( v_{t}^{n}\right)
_{n}\right) \in H
\end{equation*}
and the operator $F:H\rightarrow H$ is defined by
\begin{eqnarray}
&&F\left( 
\begin{array}{c}
x\medskip \\ 
\left( u^{n}\right) _{n}\medskip \\ 
\left( v^{n}\right) _{n}
\end{array}
\right) \medskip \medskip  \label{operatorF} \\
&=&\left( 
\begin{array}{c}
\left\langle \left( a_{n}^{1/2}n\sin \left( nx\right) \right)
_{n},\left( u^{n}\right) _{n}\right\rangle _{l^{2}}+\left\langle \left(
a_{n}^{1/2}n\cos \left( nx\right) \right) _{n},\left( v^{n}\right)
_{n}\right\rangle _{l^{2}}\medskip \\ 
\left( a_{n}^{1/2}\cos \left( nx\right) \right) _{n}\medskip \\ 
-\left( a_{n}^{1/2}\sin \left( nx\right) \right) _{n}
\end{array}
\right)  \notag
\end{eqnarray}
and $W_{t}$ is a cylindrical Wiener process with values in $H$ and the noise 
$\sigma =\left( 1,~0,~0\right) $ is the \textit{projection on the first coordinate}.

The hypotheses from this section are assumed for the rest of the paper. We
shall denote by $C$ a positive constant which might change from line to line.

\section{Existence and uniqueness of the solution for the infinite dimensional equation}

We consider the equation from the previous section%
\begin{equation}
\left\{ 
\begin{array}{l}
dY_{t}=F\left( Y_{t}\right) dt+\sigma dW_{t}\medskip \\ 
Y_{0}=y
\end{array}
\right.  \label{sist}
\end{equation}
for an initial condition $y\in \mathbb{R\times }l^{2}\times l^{2}$ and $F$
defined in (\ref{operatorF}).

We can now formulate the existence result.

\bigskip

\begin{proposition}
\label{exist}Under the assumptions presented above, for each $y\in H,$ there
is a unique analytically strong solution 
\begin{equation*}
Y\in C\left( \left[ 0,\infty \right) ;H\right) \cap L_{loc}^{\infty }\left(
0,\infty ;~H\right)
\end{equation*}
to equation (\ref{sist}).

Moreover, for $T<\infty ,$ we have that
\begin{equation*}
\mathbb{E}\left( \underset{t\in \left[ 0,T\right] }{\sup }\left\vert
Y_{t}\right\vert _{H}^{2}\right) <\infty .
\end{equation*}
\end{proposition}

\begin{proof}
We study equation (\ref{sist}) in the framework of the analytic approach of
stochastic differential equations in Banach spaces, and more precisely in
the space $H=\mathbb{R}\times l^{2}\times l^{2}$ equipped with the norm 
\begin{equation*}
\left\Vert y\right\Vert _{H}^{2}=\left\vert x\right\vert ^{2}+\left\Vert
\left( u_{n}\right) _{n}\right\Vert _{l^{2}}^{2}+\left\Vert \left(
v_{n}\right) _{n}\right\Vert _{l^{2}}^{2},\medskip
\end{equation*}
for all $y=\left( x,\left( u_{n}\right) _{n},\left( v_{n}\right) _{n}\right)
\in \mathbb{R\times }l^{2}\times l^{2}$.

Since the operator $F$ defined before is not Lipschitz in $H,$ we may use
Theorem 7.10 from page 198 of \cite{DPZ1} in order to get existence of the
solution to equation (\ref{sist}).

More precisely, we shall prove that the following three conditions are
satisfied for the operator $F$ defined in (\ref{operatorF})

\begin{itemize}
\item[a)] $F$ is locally Lipschitz continuous in $H$

\item[b)] $F$ is bounded on bounded subsets of{\large ~}$H$

\item[c)] there exists an increasing function 
\begin{equation*}
a:\mathbb{R}_{+}\rightarrow \mathbb{R}_{+},
\end{equation*}
such that 
\begin{equation*}
\left\langle F\left( y+\widetilde{y}\right) ,~y^{\ast }\right\rangle \leq
a\left( \left\Vert \widetilde{y}\right\Vert _{H}\right) \left( 1+\left\Vert
y\right\Vert _{H}\right) \medskip
\end{equation*}
for all $y,~\widetilde{y}\in H$\ and $y^{\ast }\in \partial \left\Vert
y\right\Vert $, where $\left\langle .,.\right\rangle $ is the duality form
on $H$ and $\partial \left\Vert .\right\Vert $ is the subdifferential of the 
$H$ norm.
\end{itemize}

We shall first prove a).

Indeed, for all $y$ and $\widetilde{y}$\ from $H$ we have that
\begin{equation*}
\left\Vert F\left( y\right) -F\left( \widetilde{y}\right) \right\Vert
_{H}^{2}\medskip 
\end{equation*}
\begin{equation*}
=\left\Vert \left( 
\begin{array}{c}
\sum\limits_{n}a_{n}^{1/2}n\left( \sin \left( n x\right) u_{n}+\cos
\left( nx\right) v_{n}-\sin \left( n\widetilde{x}\right) 
\widetilde{u}_{n}-\cos \left( n\widetilde{x}\right) \widetilde{v}
_{n}\right) \medskip  \\ 
\left( a_{n}^{1/2}\cos \left( n x\right) \right) _{n}-\left(
a_{n}^{1/2}\cos \left( n\widetilde{x}\right) \right) _{n}\medskip  \\ 
-\left( a_{n}^{1/2}\sin \left( nx\right) \right) _{n}+\left(
a_{n}^{1/2}\sin \left( n\widetilde{x}\right) \right) _{n}\medskip 
\end{array}
\right) \right\Vert _{H}^{2}
\end{equation*}
\begin{eqnarray*}
&=&\left\vert \sum\limits_{n}a_{n}^{1/2}n\left( \sin \left( nx\right)
u_{n}+\cos \left( nx\right) v_{n}-\sin \left( n\widetilde{x}
\right) \widetilde{u}_{n}-\cos \left( n\widetilde{x}\right) 
\widetilde{v}_{n}\right) \right\vert ^{2}
\begin{array}{c}
\medskip \medskip 
\end{array}
\\
&&\quad \quad \quad +\left\Vert \left( a_{n}^{1/2}\cos \left( nx\right)
\right) _{n}-\left( a_{n}^{1/2}\cos \left( n\widetilde{x}\right) \right)
_{n}\right\Vert _{l^{2}}^{2}
\begin{array}{c}
\medskip \medskip 
\end{array}
\\
&&\quad \quad \quad +\left\Vert \left( a_{n}^{1/2}\sin \left( nx\right)
\right) _{n}-\left( a_{n}^{1/2}\sin \left( n\widetilde{x}\right) \right)
_{n}\right\Vert _{l^{2}}^{2}
\begin{array}{c}
\medskip \medskip 
\end{array}
\end{eqnarray*}
\begin{equation}
\overset{Denote}{=}T_{1}+T_{2}+T_{3}.\medskip \quad \quad \quad \quad \quad
\quad \quad \quad \quad \quad \quad \quad \quad \quad \quad \quad \quad
\quad \quad \quad \quad \quad \quad \quad   \label{loc}
\end{equation}

For the first term we see that 
\begin{eqnarray*}
T_{1} &\leq &2\left\vert \sum\limits_{n}a_{n}^{1/2}n\left( \sin \left(
nx\right) u_{n}-\sin \left( n\widetilde{x}\right) \widetilde{u}
_{n}\right) \right\vert ^{2}
\begin{array}{c}
\medskip \medskip\medskip 
\end{array}
\\
&&+2\left\vert \sum\limits_{n}a_{n}^{1/2}n\left( \cos \left( nx\right)
v_{n}-\cos \left( n\widetilde{x}\right) \widetilde{v}_{n}\right)
\right\vert ^{2}\medskip 
\end{eqnarray*}
\begin{equation*}
\leq 4\left\vert \sum\limits_{n}a_{n}^{1/2}n\sin \left( nx\right)
\left( u_{n}-\widetilde{u}_{n}\right) \right\vert ^{2}+4\left\vert
\sum\limits_{n}a_{n}^{1/2}n\left( \sin \left( nx\right) -\sin \left( n
\widetilde{x}\right) \right) \widetilde{u}_{n}\right\vert
^{2}\medskip 
\end{equation*}
\begin{equation*}
+4\left\vert \sum\limits_{n}a_{n}^{1/2}n\cos \left( nx\right) \left(
v_{n}-\widetilde{v}_{n}\right) \right\vert ^{2}+4\left\vert
\sum\limits_{n}a_{n}^{1/2}n\left( \cos \left( nx\right) -\cos \left( n%
\widetilde{x}\right) \right) \widetilde{v}_{n}\right\vert
^{2}\medskip 
\end{equation*}
and then, by the Cauchy-Schwarz inequality for the inner product in $l^{2}$ and
taking into account that $\left( a_{n}\right) _{n}\in O^{5}$, we obtain that 
\begin{eqnarray*}
T_{1} &\leq &C\left\Vert \left( a_{n}^{1/2}n\sin \left( nx\right)
\right) _{n}\right\Vert _{l^{2}}^{2}\left\Vert \left( u_{n}\right)
_{n}-\left( \widetilde{u}_{n}\right) _{n}\right\Vert _{l^{2}}^{2}
\begin{array}{c}
\medskip \medskip \medskip 
\end{array}
\\
&&+C\left\Vert \left( a_{n}^{1/2}n\cos \left( nx\right) \right)
_{n}\right\Vert _{l^{2}}^{2}\left\Vert \left( v_{n}\right) _{n}-\left( 
\widetilde{v}_{n}\right) _{n}\right\Vert _{l^{2}}^{2}%
\begin{array}{c}
\medskip \medskip \medskip
\end{array}
\\
&&+\left( (\sum\limits_{n}a_{n}^{1/2}n^{2}\left\vert\widetilde{u}
_{n}\right\vert) ^{2}+(\sum\limits_{n}a_{n}^{1/2}n^{2}
\left\vert\widetilde{v}_{n}\right\vert )^{2}\right) \left\vert x-\widetilde{x}
\right\vert ^{2}
\begin{array}{c}
\medskip \medskip \medskip
\end{array}
\\
&\leq &C\left\Vert \left( u_{n}\right) _{n}-\left( \widetilde{u}
_{n}\right) _{n}\right\Vert _{l^{2}}^{2}+C\left\Vert \left(
v_{n}\right) _{n}-\left( \widetilde{v}_{n}\right) _{n}\right\Vert
_{l^{2}}^{2}
\begin{array}{c}
\medskip \medskip \medskip
\end{array}
\\
&&+C\left( \left\Vert \left( \widetilde{u}_{n}\right) \right\Vert
_{l^{2}}^{2}+\left\Vert \left( \widetilde{v}_{n}\right) \right\Vert
_{l^{2}}^{2}\right) \left\vert x-\widetilde{x}\right\vert ^{2}
\begin{array}{c}
\medskip \medskip \medskip
\end{array}
\\
&\leq &C(1+\left\Vert \left( \widetilde{u}_{n}\right) \right\Vert
_{l^{2}}^{2}+\left\Vert \left( \widetilde{v}_{n}\right) \right\Vert
_{l^{2}}^{2})
\begin{array}{c}
\medskip \medskip \medskip
\end{array}
\\
&&\times \left( \left\vert x-\widetilde{x}\right\vert
^{2}+\left\Vert \left( u_{n}\right) _{n}-\left( \widetilde{u}%
_{n}\right) _{n}\right\Vert _{l^{2}}^{2}+\left\Vert \left(
v_{n}\right) _{n}-\left( \widetilde{v}_{n}\right) _{n}\right\Vert
_{l^{2}}^{2}\right),
\begin{array}{c}
\medskip  
\end{array}
\end{eqnarray*}
which leads to 
\begin{equation*}
T_{1}\leq C(1+\left\Vert \left( \widetilde{u}_{n}\right) \right\Vert
_{l^{2}}^{2}+\left\Vert \left( \widetilde{v}_{n}\right) \right\Vert
_{l^{2}}^{2})\left\Vert y-\widetilde{y}\right\Vert _{H}^{2}%
\begin{array}{c}
\medskip \medskip \medskip 
\end{array}
\end{equation*}
where $C$ is a positive constant depending on $\left( a_{n}\right) _{n}$
which might change from line to line.

Keeping in mind that $\left( a_{n}\right) _{n}\in O^{5},$ we can easily see
that the second and the third term verify
\begin{eqnarray*}
T_{2} &=&\sum_{n}\left\vert a_{n}^{1/2}\left( \cos \left( nx\right)
-\cos \left( n\widetilde{x}\right) \right) \right\vert ^{2}%
\begin{array}{c}
\medskip \medskip \medskip
\end{array}
\\
&\leq &\sum_{n}\left\vert a_{n}^{1/2}n\left( x-\widetilde{x}\right)
\right\vert ^{2}
\begin{array}{c}
\medskip \medskip \medskip
\end{array}
\\
&\leq &\sum_{n}a_{n}n^{2}\left\vert x-\widetilde{x}\right\vert ^{2}
\begin{array}{c}
\medskip \medskip \medskip
\end{array}
\\
&\leq &C\left\vert x-\widetilde{x}\right\vert ^{2}
\end{eqnarray*}
and, by a similar argument, 
\begin{equation*}
T_{3}\leq C\left\vert x-\widetilde{x}\right\vert ^{2}.
\end{equation*}

Going back to (\ref{loc}) we obtain that
\begin{eqnarray}\label{constantC}
\left\Vert F\left( y\right) -F\left( \widetilde{y}\right) \right\Vert
_{H}^{2} &\leq &C(1+\left\Vert \left( \widetilde{u}_{n}\right)
\right\Vert _{l^{2}}^{2}+\left\Vert \left( \widetilde{v}_{n}\right)
\right\Vert _{l^{2}}^{2})\left\Vert y-\widetilde{y}\right\Vert _{H}^{2}
\begin{array}{c}
\medskip \medskip \medskip
\end{array}
\label{estimare} \\
&\leq &C\left( 1+\left\Vert \widetilde{y}\right\Vert _{H}^{2}\right)
\left\Vert y-\widetilde{y}\right\Vert _{H}^{2}
\begin{array}{c}
\medskip \medskip
\end{array}
\notag
\end{eqnarray}
where $C$ is a positive constant depending on $\left( a_{n}\right) _{n}$.

Consequently, for all $y,\widetilde{y}\in B\left( 0,R\right) $ we obtain that
\begin{eqnarray}
\left\Vert F\left( y\right) -F\left( \widetilde{y}\right) \right\Vert _{H}
&\leq &C\sqrt{(1+\left\Vert \left( \widetilde{u}_{n}\right) \right\Vert
_{l^{2}}^{2}+\left\Vert \left( \widetilde{v}_{n}\right) \right\Vert
_{l^{2}}^{2})}\left\Vert y-\widetilde{y}\right\Vert _{H}%
\begin{array}{c}
\medskip \medskip 
\end{array}
\label{lipschitz} \\
&\leq &\sqrt{C\left( 1+\left\Vert \widetilde{y}\right\Vert _{H}^{2}\right) }
\left\Vert y-\widetilde{y}\right\Vert _{H}
\begin{array}{c}
\medskip \medskip
\end{array}
\notag \\
&\leq &C\left( R,\left( a_{n}\right) _{n}\right) \left\Vert y-\widetilde{y}
\right\Vert _{H}
\begin{array}{c}
\medskip \medskip
\end{array}
\notag
\end{eqnarray}%
\ where $C\left( R,\left( a_{n}\right) _{n}\right) $ is a positive constant
depending on $R$ and $\left( a_{n}\right) _{n}$, and the proof of the
locally Lipschitz property is completed.

For the proof of b) it is sufficient to take $\widetilde{y}=0$\ in (\ref{estimare}). We obtain then 
\begin{eqnarray}
\left\Vert F\left( y\right) \right\Vert _{H} &\leq &\left\Vert F\left(
y\right) -F\left( 0\right) \right\Vert _{H}+\left\Vert F\left( 0\right)
\right\Vert _{H}
\begin{array}{c}
\medskip \medskip
\end{array}
\label{sublin} \\
&\leq &C\left\Vert y\right\Vert _{H}+\left\Vert \left( a_{n}^{1/2}\right)
_{n}\right\Vert _{l^{2}}
\begin{array}{c}
\medskip \medskip
\end{array}
\notag \\
&\leq &C\left( \left\Vert y\right\Vert _{H}+1\right) 
\begin{array}{c}
\medskip \medskip
\end{array}
\notag
\end{eqnarray}
where $C$ is a positive constant depending on $\left( a_{n}\right) _{n}$
which might change from line to line. Consequently, $F$ is bounded on
bounded subsets of $H$.

In order to complete the proof of existence, we still have to prove c) and
to this purpose we need to find an increasing function 
\begin{equation*}
a:\mathbb{R}_{+}\rightarrow \mathbb{R}_{+},
\end{equation*}
such that 
\begin{equation*}
\left\langle F\left( y+\widetilde{y}\right) ,~y^{\ast }\right\rangle \leq
a\left( \left\Vert \widetilde{y}\right\Vert _{H}\right) \left( 1+\left\Vert
y\right\Vert _{H}\right) \medskip
\end{equation*}
for all $y,~\widetilde{y}\in H$\ and $y^{\ast }\in \partial \left\Vert
y\right\Vert $.

For that purpose, we consider the function $a(\alpha)=C\left( 1+\alpha
\right) $, where $C$ is the constant from (\ref{constantC}).The constant being positive, the function is clearly increasing on $\mathbb{R}_{+}$. 

Since the subdifferential of the application 
\begin{equation*}
y\rightarrow \frac{1}{2}\left\Vert y\right\Vert _{H}^{2}
\end{equation*}
is the duality mapping of the space $H$, and in our case $H=H^{\ast }$, we
have that
\begin{equation*}
\partial \left\Vert y\right\Vert _{H}=\left\{ 
\begin{array}{l}
\left\{ \dfrac{y}{\left\Vert y\right\Vert _{H}}\right\},\text{ for }y\neq 0
\begin{array}{c}
\medskip \medskip \medskip
\end{array}
\\ 
\left\{ \left\Vert y\right\Vert _{H}\leq 1\right\} ,\text{ for }y=0
\begin{array}{c}
\medskip \medskip \medskip
\end{array}
\end{array}
\right. ,
\end{equation*}
(see page 72 from \cite{DPZ2}).

Since the case $y=0$ is trivial, we only need to prove that 
\begin{equation*}
\left\langle F\left( y+\widetilde{y}\right) ,~\dfrac{y}{\left\Vert
y\right\Vert _{H}}\right\rangle _{H}\leq a\left( \left\Vert \widetilde{y}
\right\Vert _{H}\right) \left( 1+\left\Vert y\right\Vert _{H}\right) .
\begin{array}{c}
\medskip \medskip \medskip
\end{array}
\end{equation*}

Indeed, for all $y=\left( x,\left( u^{n}\right) _{n},\left( v^{n}\right)
_{n}\right) $ and $\widetilde{y}=\left( \widetilde{x},\left( \widetilde{u}
^{n}\right) _{n},\left( \widetilde{v}^{n}\right) _{n}\right) $ in $H$ we
have that 
\begin{eqnarray*}
\left\langle F\left( y+\widetilde{y}\right) ,~\dfrac{y}{\left\Vert
y\right\Vert _{H}}\right\rangle _{H} &\leq &\left\langle F\left( y+
\widetilde{y}\right) -F\left( \widetilde{y}\right) ,~\dfrac{y}{\left\Vert
y\right\Vert _{H}}\right\rangle _{H}
\begin{array}{c}
\medskip \medskip \medskip 
\end{array}
\\
&&+\left\langle F\left( \widetilde{y}\right) ,~\dfrac{y}{\left\Vert
y\right\Vert _{H}}\right\rangle _{H}
\begin{array}{c}
\medskip \medskip \medskip 
\end{array}
\\
&\leq &C\sqrt{\left( 1+\left\Vert \widetilde{y}\right\Vert _{H}^{2}\right) }
\left\Vert y\right\Vert _{H}+C\left( \left\Vert \widetilde{y}\right\Vert
_{H}+1\right) 
\begin{array}{c}
\medskip \medskip \medskip 
\end{array}
\\
&\leq &C\left( 1+\left\Vert \widetilde{y}\right\Vert _{H}\right) \left(
1+\left\Vert y\right\Vert _{H}\right) 
\begin{array}{c}
\medskip \medskip \medskip 
\end{array}
\end{eqnarray*}
where $C$ is a positive constant depending only on $\left( a_{n}\right) _{n}$
which might change from line to line. Hence, we obtain 
\begin{equation*}
\left\langle F\left( y+\widetilde{y}\right) ,~\dfrac{y}{\left\Vert
y\right\Vert }\right\rangle _{H}\leq a\left( \left\Vert \widetilde{y}
\right\Vert _{H}\right) \left( 1+\left\Vert y\right\Vert _{H}\right).
\end{equation*}

We have now existence of an unique mild solution. Since in our case the
generator of $C_{0}$-semigroup is identically zero, a solution is strong if and only if it is mild (see \cite{concise}), so we have also existence and uniqueness of a
strong solution.

Consequently, the proof of existence and uniqueness
is complete.

We shall now prove that
\begin{equation*}
\mathbb{E}\left( \underset{t\in \left[ 0,T\right] }{\sup }\left\Vert
Y_{t}\right\Vert _{H}^{2}\right) <\infty .
\end{equation*}

To this purpose we apply the It\^{o} formula to equation (\ref{sist}) with
the function 
\begin{equation*}
y\mapsto \frac{1}{2}\left\Vert y\right\Vert _{H}^{2}
\end{equation*}
and we get
\begin{eqnarray}
\frac{1}{2}\left\Vert Y\left( t\right) \right\Vert _{H}^{2} &=&\frac{1}{2}
\left\Vert y\right\Vert _{H}^{2}+\int_{0}^{t}\left\langle F\left( Y\left(
s\right) \right) ,Y\left( s\right) \right\rangle _{H}ds
\begin{array}{c}
\medskip \medskip \medskip
\end{array}
\label{estt} \\
&&+\int_{0}^{t}\left\langle Y\left( s\right) ,\sigma dW_{s}\right\rangle _{H}+
\frac{1}{2}\int_{0}^{t}\left\vert \sigma \right\vert ^{2}ds.
\begin{array}{c}
\medskip \medskip \medskip
\end{array}
\notag
\end{eqnarray}

We can easily see that 
\begin{eqnarray*}
\int_{0}^{t}\left\langle Y\left( s\right) ,\sigma dW_{s}\right\rangle _{H}
&=&\int_{0}^{t}X\left( s\right) d\beta_{s}
\begin{array}{c}
\medskip \medskip \medskip
\end{array}
\\
&\leq &\underset{t\in \left[ 0,T\right] }{\sup }\left\vert
\int_{0}^{t}X\left( s\right) d\beta_{s}\right\vert 
\begin{array}{c}
\medskip \medskip \medskip
\end{array}
\end{eqnarray*}
and then, by using the Burkholder-Davis-Gundy inequality, we obtain that 
\begin{equation*}
\mathbb{E}\left( \underset{r\in \left[ 0,t\right] }{\sup }\left\vert
\int_{0}^{r}X\left( s\right) d\beta_{s}\right\vert \right) \leq C\mathbb{E}
\left( \int_{0}^{t}\left\vert X\left( s\right) \right\vert ^{2}ds\right)
^{1/2}
\begin{array}{c}
\medskip \medskip \medskip
\end{array}
\end{equation*}
(see, e.g., \cite{DPZ2} page 58).

On the other hand we see that, by (\ref{sublin}), we get that
\begin{eqnarray*}
\left\langle F\left( Y\left( s\right) \right) ,Y\left( s\right)
\right\rangle _{H} &\leq &\left\Vert F\left( Y\left( s\right) \right)
\right\Vert _{H}\left\Vert Y\left( s\right) \right\Vert _{H}
\begin{array}{c}
\medskip \medskip  
\end{array}
\\
&\leq &C\left( \left\Vert Y\left( s\right) \right\Vert _{H}+1\right)
\left\Vert Y\left( s\right) \right\Vert _{H}%
\begin{array}{c}
\medskip \medskip  
\end{array}
\\
&\leq &C\left( 1+\left\Vert Y\left( s\right) \right\Vert _{H}^{2}\right) ,
\begin{array}{c}
\medskip \medskip 
\end{array}
\end{eqnarray*}
where $C$ is a positive constant depending only on $\left( a_{n}\right) _{n}$
that changes from line to line.

By going back into (\ref{estt}) we obtain via the estimates above that
\begin{equation*}
\mathbb{E}\left( \underset{r\in \left[ 0,t\right] }{\sup }\left\Vert Y\left(
r\right) \right\Vert _{H}^{2}\right) \leq \left\Vert y\right\Vert _{H}^{2}+C
\mathbb{E}\int_{0}^{t}\left( \underset{r\in \left[ 0,s\right] }{\sup }
\left\Vert Y\left( r\right) \right\Vert _{H}^{2}\right) ds+Ct,
\begin{array}{c}
\medskip \medskip 
\end{array}
\end{equation*}
and finally, by Gronwall's lemma we obtain 
\begin{equation*}
\mathbb{E}\left( \underset{t\in \left[ 0,T\right] }{\sup }\left\Vert
Y_{t}\right\Vert _{H}^{2}\right) \leq Ce^{CT}\left( \left\Vert y\right\Vert
_{H}^{2}+T\right) <\infty 
\end{equation*}
and the proof is now complete.
\end{proof}

\begin{remark}
Note that the solution obtained above has the Markov property. For details
see Theorem 9.8 from \cite{DPZ1}.
\end{remark}

\section{The Feller property of the transition semigroup}

We consider the transition semigroup corresponding to the solution $Y\left(
t,y\right) $ defined by
\begin{equation*}
P_{t}\varphi \left( y\right) =\mathbb{E}\left[ \varphi \left( Y\left(
t,y\right) \right) \right] ,
\end{equation*}
for all $\varphi \in B_{b}\left( H\right) $, the space of all bounded and
Borel real functions in $H$, for all$~t\geq 0$ and for all $y\in H$.

We intend to prove that the semigroup has the Feller property which means
that it maps bounded continuous functions into bounded continuous functions.

\begin{proposition}
\label{feller}Let $\left( y_{k}\right) _{k\in \mathbb{N}}$ be a sequence of
initial conditions from $H$ such that $y_{k}\rightarrow y$ in $H$ for $k\rightarrow \infty$. If we denote by 
\begin{equation*}
Y_{k}\left( t\right) =\left(
X_{t}^{k},~\left( u_{t}^{n}\right) _{n}^{k},~\left( v_{t}^{n}\right)
_{n}^{k}\right) 
\end{equation*}
and 
\begin{equation*} 
Y\left( t\right) =\left( X_{t},~\left(u_{t}^{n}\right) _{n},~\left( v_{t}^{n}\right) _{n}\right )
\end{equation*}
the solutions to equation (\ref{sist}) corresponding to every $y_{k}$ and respectively to $y$, then, for any $t>0$, we have that 
\begin{equation*}
\left\Vert Y_{k}\left( t\right) -Y\left( t\right) \right\Vert
_{H}^{2}\rightarrow 0\quad \text{as }k\rightarrow \infty .\medskip
\end{equation*}
In particular we have also that $\left( Y_{t}\right) _{t\geq 0}$ is a Feller process.
\end{proposition}

\begin{proof}
We shall check first the following a priori estimates.

Since 
\begin{equation*}
u_{t}^{n}=u_{0}^{n}+a_{n}^{1/2}\int_{0}^{t}\cos \left( nX_{s}\right)
ds\medskip
\end{equation*}
we can easily obtain that
\begin{eqnarray}
\left\Vert \left( u_{t}^{n}\right) _{n}\right\Vert _{l^{2}}^{2} &\leq
&2\left\Vert \left( u_{0}^{n}\right) _{n}\right\Vert
_{l^{2}}^{2}+2t^{2}\left\Vert \left( a_{n}^{1/2}\right) _{n}\right\Vert
_{l^{2}}^{2}
\begin{array}{c}
\medskip \medskip \medskip
\end{array}
\label{est1} \\
&\leq &C\left( \left( u_{0}^{n}\right) _{n}\right) (1+t)^{2},
\begin{array}{c}
\medskip \medskip \medskip
\end{array}
\notag
\end{eqnarray}
where $C\left( \left( u_{0}^{n}\right) _{n}\right) $ is a constant\ which
might change from line to line,\ depending on the initial condition $\left(
u_{0}^{n}\right) _{n}.$

Of course, by the same argument, we get that 
\begin{equation}
\left\Vert \left( v_{t}^{n}\right) _{n}\right\Vert _{l^{2}}^{2}\leq C\left(
\left( v_{0}^{n}\right) _{n}\right) (1+t)^{2}.
\begin{array}{c}
\medskip \medskip 
\end{array}
\label{est2}
\end{equation}

By taking the inner product in $H$\ between the difference%
\begin{equation*}
\frac{d}{dt}\left( Y_{k}\left( t\right) -Y\left( t\right) \right) =F\left(
Y_{k}\left( t\right) \right) -F\left( Y\left( t\right) \right) 
\begin{array}{c}
\medskip \medskip 
\end{array}
\end{equation*}
and $\left( Y_{k}\left( t\right) -Y\left( t\right) \right) $ and keeping in
mind that 
\begin{equation*}
\left\langle \frac{d}{dt}\left( Y_{k}\left( t\right) -Y\left( t\right)
\right) ,\left( Y_{k}\left( t\right) -Y\left( t\right) \right) \right\rangle
_{H}=\frac{d}{dt}\left( \frac{1}{2}\left\Vert \left( Y_{k}\left( t\right)
-Y\left( t\right) \right) \right\Vert _{H}^{2}\right) ,
\end{equation*}
we get that
\begin{eqnarray}
\left\Vert Y_{k}\left( t\right) -Y\left( t\right) \right\Vert _{H}^{2}
&=&\left\Vert y_{k}-y\right\Vert _{H}^{2}
\begin{array}{c}
\medskip \medskip \medskip
\end{array}
\label{ito} \\
&&+2\int_{0}^{t}\left\langle F\left( Y_{k}\left( s\right) \right) -F\left(
Y\left( s\right) \right) ,Y_{k}\left( s\right) -Y\left( s\right)
\right\rangle _{H}ds.
\begin{array}{c}
\medskip \medskip
\end{array}
\notag
\end{eqnarray}

We can see by (\ref{lipschitz}) that%
\begin{eqnarray*}
&&\left\langle F\left( Y_{k}\left( s\right) \right) -F\left( Y\left(
s\right) \right) ,Y_{k}\left( s\right) -Y\left( s\right) \right\rangle _{H}
\begin{array}{c}
\medskip \medskip \medskip
\end{array}
\\
&\leq &\left\Vert F\left( Y_{k}\left( s\right) \right) -F\left( Y\left(
s\right) \right) \right\Vert _{H}\left\Vert Y_{k}\left( s\right) -Y\left(
s\right) \right\Vert _{H}
\begin{array}{c}
\medskip \medskip \medskip
\end{array}
\\
&\leq &C\sqrt{(1+\left\Vert \left( u_{s}^{n}\right) \right\Vert
_{l^{2}}^{2}+\left\Vert \left( v_{s}^{n}\right) \right\Vert _{l^{2}}^{2})}
\left\Vert Y_{k}\left( s\right) -Y\left( s\right) \right\Vert _{H}^{2}
\begin{array}{c}
\medskip \medskip \medskip
\end{array}
\end{eqnarray*}
and then, by (\ref{est1}) and (\ref{est2}) we see that
\begin{eqnarray*}
&&\left\langle F\left( Y_{k}\left( s\right) \right) -F\left( Y\left(
s\right) \right) ,Y_{k}\left( s\right) -Y\left( s\right) \right\rangle _{H}
\begin{array}{c}
\medskip \medskip \medskip
\end{array}
\\
&\leq &C\left( \left( u_{0}^{n}\right) _{n},\left( v_{0}^{n}\right)
_{n}\right) \left( 1+s\right) \left\Vert Y_{k}\left( s\right) -Y\left(
s\right) \right\Vert _{H}^{2},
\begin{array}{c}
\medskip \medskip \medskip
\end{array}
\end{eqnarray*}
where $C$ is a positive constant which might depend on $\left( a_{n}\right)
_{n}$ and also on the initial condition $y=\left( x,\left( u_{0}^{n}\right)
_{n},\left( v_{0}^{n}\right) _{n}\right) $.

Finally, from (\ref{ito}) we have that
\begin{eqnarray*}
&&\left\Vert Y_{k}\left( t\right) -Y\left( t\right) \right\Vert
_{H}^{2}
\begin{array}{c}
\medskip \medskip \medskip 
\end{array}
\\
&=&\left\Vert y_{k}-y\right\Vert _{H}^{2}+2\int_{0}^{t}\left
\langle F\left( Y_{k}\left( s\right) \right) -F\left( Y\left( s\right)
\right) ,Y_{k}\left( s\right) -Y\left( s\right) \right\rangle _{H}ds
\begin{array}{c}
\medskip \medskip \medskip 
\end{array}
\\
&\leq &\left\Vert y_{k}-y\right\Vert _{H}^{2}+C\left( y\right)
\int_{0}^{t}\left( 1+s\right)\left\Vert Y_{k}\left( s\right)
-Y\left( s\right) \right\Vert _{H}^{2}ds
\begin{array}{c}
\medskip \medskip \medskip 
\end{array}
\end{eqnarray*}
where $C$ is a positive constant depending on $\left( a_{n}\right) _{n}$ and
also on the initial condition $y=\left( x,\left( u_{0}^{n}\right)
_{n},\left( v_{0}^{n}\right) _{n}\right) $.

Then, by Gronwall's lemma, we obtain that
\begin{equation*}
\left\Vert Y_{k}\left( t\right) -Y\left( t\right) \right\Vert
_{H}^{2}\leq e^{C\left( y\right) \left( t+t^{2}\right) }\left\Vert
y_{k}-y\right\Vert _{H}^{2}.
\begin{array}{c}
\end{array}
\end{equation*}

Let $\varphi :H\rightarrow \mathbb{R}$ be a bounded and continuous function. Since $L^{2}$ convergence implies a convergence in probability, we then have that
$\varphi (Y_{k}(t))\rightarrow \varphi ( Y(t))$ in probability (see Lemma 3.3 in \cite{Kallenberg}). 


Consequently, 
\begin{equation*}
\underset{k\rightarrow \infty }{\lim }\mathbb{E\varphi }\left( Y_{k}\left(
t\right) \right) =\mathbb{E\varphi }\left( Y\left( t\right) \right) ,\quad 
\text{for any fixed }t>0,
\begin{array}{c}
\medskip \medskip
\end{array}
\end{equation*}
which is actually
\begin{equation*}
\underset{k\rightarrow \infty }{\lim }P_{t}\mathbb{\varphi }\left(
y_{k}\right) =P_{t}\mathbb{\varphi }\left( y\right) ,\quad \text{for any
fixed }t>0,
\begin{array}{c}
\medskip \medskip
\end{array}%
\end{equation*}%
and then we have the proved the Feller property.
\end{proof}

\begin{remark}\label{noStrongFell}
 Let $\mathcal{A}=\mathbb{R}\times O^{1}\times O^{1}$, with $O^{1}$ defined by (\ref{space}). It easily follows from the definition of $(u_{t}^{n})_{n}$ and $(v_{t}^{n})_{n}$ that $$ y\in \mathcal{A} \Leftrightarrow Y(t)\in\mathcal{A} \text{ for all } t\geqslant 0, $$
where $Y(t)$ is the solution of equation (\ref{sist}) with initial condition $Y(0)=y$. This makes $1_{\mathcal{A}}$ invariant under $P_{t}$ (i.e., $P_{t}1_{\mathcal{A}}=1_{\mathcal{A}}$).\\
Hence the process $\left( Y_{t}\right) _{t}$ is not strongly Feller.
\end{remark}

\section{The invariant measure of the transition semigroup}\label{s5}

In this section we shall prove existence of an invariant measure for the 
transition semigroup corresponding to the equation on $\mathbb{S}^{1}\times l^{2}\times l^{2}$
\begin{equation}
\left\{ 
\begin{array}{l}
dY_{t}=F\left( Y_{t}\right) dt+\sigma dW_{t}\medskip \\ 
Y_{0}=y
\end{array}
\right.  \label{ecuatie}
\end{equation}
with initial condition $y\in \mathbb{R\times }l^{2}\times l^{2}$, where $\mathbb{S}^{1}$ is identified to $\mathbb{R}/2\pi\mathbb{Z}$.

A probability $\mu $ on $H$ is said to be an invariant measure for the transition semigroup $
\left( P_{t}\right) _{t}$ iff 
\begin{equation}
\int_{H}P_{t}\mathbb{\varphi }\left( y\right) \mu \left( dy\right) =\int_{H}
\mathbb{\varphi }\left( y\right) \mu \left( dy\right),  \label{inv}
\end{equation}
for all measurable and bounded function $\varphi$.

By standard arguments (see Theorem 1.2, page 8 from \cite{Billing} and relation (1.5) at page 2 of \cite{DP}) it is sufficient that (\ref{inv}) holds for all $\varphi\in C_{b}(H)$.

\bigskip

\textbf{Existence of an invariant measure of the transition semigroup}

\bigskip

We consider the measure
\begin{equation}
\mu \left( dy\right) =\frac{dx}{2\pi }\otimes \prod\limits_{n\geq 1}N\left(
0,\frac{1}{n^{2}}\right) du_{n}\otimes \prod\limits_{n\geq 1}N\left( 0,
\frac{1}{n^{2}}\right) dv_{n}  \label{invar}
\end{equation}
where $N\left( 0,\frac{1}{n^{2}}\right) $\ is the normal distribution. The
form of $\mu $ is inspired from the finite dimensional case (see \cite{CEGMB}).

First, the fact that $\mu $ is a probability measure on H is clearly explained in 
Exercise 2.1.8. from \cite{concise}.

We intend to prove that $\mu $ is an invariant measure of $\left(
P_{t}\right) _{t}$ on $\mathbb{S}^{1}\times l^{2}\times l^{2}$ by using the strong convergence of a Galerkin type
approximation.

To this purpose, we consider that 
\begin{equation*}
H=H_{N}\times l^{2}\times l^{2}
\end{equation*}
where $H_{N}=\mathbb{R}\times \mathbb{R}^{N}\times \mathbb{R}^{N},$ and 
\begin{equation*}
\Pi _{N}:H\rightarrow H_{N}\times \left\{ 0\right\} ^{\infty }\times \left\{
0\right\} ^{\infty }
\end{equation*}
be defined by 
\begin{equation*}
\Pi _{N}\left( x,\left( u_{n}\right) _{n\in \mathbb{N}},\left( v_{n}\right)
_{n\in \mathbb{N}}\right) =\left( x,\left( u_{n}\right) _{n=1}^{N}\times
\left\{ 0\right\} ^{\infty },\left( u_{n}\right) _{n=1}^{N}\times \left\{
0\right\} ^{\infty }\right) .
\end{equation*}

Obviously, the following stochastic equation on $H_{N}$ 
\begin{equation}
\left\{ 
\begin{array}{l}
dY_{t}^{\left( N\right) }=\Pi _{N}\left( F\left( Y_{t}^{\left( N\right)
}\right) \right) dt+\sigma dW_{t}\medskip  \\ 
Y_{0}^{\left( N\right) }=\Pi _{N}y\medskip 
\end{array}
\right.   \label{aprox}
\end{equation}
can be treated by classical results for the solvability of SDE in
finite-dimension. Consequently, equation (\ref{aprox}) has a unique strong
solution.

\bigskip

We can now prove the following preliminary result.

\begin{lemma}\label{CVGE}
Under the assumptions given before, the sequence of solutions $\left(
Y^{\left( N\right) }\right) _{N}$ to equations (\ref{aprox}) converges
strongly in $H$ to the solution $Y$ to equation (\ref{sist}). 
More precisely we have that
\begin{equation*}
\lim_{N\rightarrow \infty}\sup_{0\leqslant t\leqslant T}\left\Vert Y^{\left( N\right) }\left(
t\right) -Y\left( t\right) \right\Vert _{H}^{2}=0,
\end{equation*}
for all $T>0$ and $\omega \in \Omega .$
\end{lemma}

\begin{proof}
By taking the inner product between 
$Y^{\left( N\right) }\left( t\right) -Y\left( t\right)$ 
and the difference 
\begin{equation*}
\frac{d}{dt}\left( Y^{\left( N\right) }\left( t\right) -Y\left( t\right)
\right) =\left( \Pi _{N}F\left( Y^{\left( N\right) }\left( s\right) \right)
-F\left( Y\left( s\right) \right) \right) \medskip 
\end{equation*}
we obtain that
\begin{eqnarray*}
&&\left\Vert Y^{\left( N\right) }\left( t\right) -Y\left( t\right)
\right\Vert _{H}^{2}
\begin{array}{c}
\medskip \medskip \medskip 
\end{array}
\\
&=&\left\Vert \Pi _{N}y-y\right\Vert _{H}^{2}
\begin{array}{c}
\medskip \medskip
\end{array}
\\
&&+2\int_{0}^{t}\left\langle \Pi _{N}F\left( Y^{\left( N\right) }\left(
s\right) \right) -F\left( Y\left( s\right) \right) ,Y^{\left( N\right)
}\left( s\right) -Y\left( s\right) \right\rangle _{H}ds
\begin{array}{c}
\medskip \medskip 
\end{array}
\\
&=&\left\Vert \Pi _{N}y-y\right\Vert _{H}^{2}
\begin{array}{c}
\medskip \medskip 
\end{array}
\\
&&+2\int_{0}^{t}\left\langle \Pi _{N}F\left( Y^{\left( N\right) }\left(
s\right) \right) -F\left( Y^{\left( N\right) }\left( s\right) \right)
,Y^{\left( N\right) }\left( s\right) -Y\left( s\right) \right\rangle _{H}ds
\begin{array}{c}
\medskip \medskip \medskip
\end{array}
\\
&&+2\int_{0}^{t}\left\langle F\left( Y^{\left( N\right) }\left( s\right)
\right) -F\left( Y\left( s\right) \right) ,Y^{\left( N\right) }\left(
s\right) -Y\left( s\right) \right\rangle _{H}ds.
\begin{array}{c}
\medskip \medskip \medskip  
\end{array}
\end{eqnarray*}

We can easily see that 
\begin{eqnarray*}
&&\left\langle \Pi _{N}F\left( Y^{\left( N\right) }\left( s\right) \right)
-F\left( Y^{\left( N\right) }\left( s\right) \right) ,Y^{\left( N\right)
}\left( s\right) -Y\left( s\right) \right\rangle _{H}%
\begin{array}{c}
\medskip \medskip \medskip 
\end{array}
\\
&\leq &\left\Vert \left( 0,\left( a_{n}^{1/2}\cos \left( nX^{(N)}\left( s\right)
\right) \right) _{n>N},-\left( a_{n}^{1/2}\sin \left( nX^{(N)}\left( s\right)
\right) \right) _{n>N}\right) \right\Vert _{H}%
\begin{array}{c}
\medskip \medskip \medskip 
\end{array}
\\
&&\times \left\Vert Y^{\left( N\right) }\left( s\right) -Y\left( s\right)
\right\Vert _{H}
\begin{array}{c}
\medskip \medskip \medskip
\end{array}
\\
&\leq &C\left\Vert \left( a_{n}^{1/2}\right) _{n>N}\right\Vert
_{l^{2}}^{2}+\left\Vert Y^{\left( N\right) }\left( s\right) -Y\left(
s\right) \right\Vert _{H}^{2}
\begin{array}{c}
\medskip \medskip \medskip 
\end{array}
\end{eqnarray*}
and, by arguing as in Proposition \ref{feller}, we have that 
\begin{eqnarray*}
&&\left\langle F\left( Y^{\left( N\right) }\left( s\right) \right) -F\left(
Y\left( s\right) \right) ,Y^{\left( N\right) }\left( s\right) -Y\left(
s\right) \right\rangle _{H}
\begin{array}{c}
\medskip \medskip \medskip 
\end{array}
\\
&\leq &C\left( y\right) \left( 1+s\right) \left\Vert Y^{\left( N\right)
}\left( s\right) -Y\left( s\right) \right\Vert _{H}^{2}.
\begin{array}{c}
\medskip \medskip \medskip
\end{array}
\end{eqnarray*}
where $C$ is a positive constant depending on $\left( a_{n}\right) _{n}$ and
also on the initial condition 
\medskip\medskip 
$y=\left( x,\left( u_{0}^{n}\right)_{n},\left( v_{0}^{n}\right) _{n}\right) $.

We obtain, for $0\leqslant t\leqslant T$, that
\begin{eqnarray*}
\left\Vert Y^{\left( N\right) }\left( t\right) -Y\left( t\right) \right\Vert
_{H}^{2} &\leq &\left\Vert \Pi _{N}y-y\right\Vert _{H}^{2}+Ct\left\Vert
\left( a_{n}^{1/2}\right) _{n>N}\right\Vert _{l^{2}}^{2}%
\begin{array}{c}
\medskip \medskip \medskip
\end{array}
\\
&&+C\left( y\right) \int_{0}^{t}\left( 1+s\right) \left\Vert Y^{\left(
N\right) }(s)-Y(s)\right\Vert _{H}^{2}ds.\\
\begin{array}{c}
\medskip \medskip
\end{array}
&\leq & \Vert\Pi_{N}y-y\Vert_{H}^{2}+CT\Vert
( a_{n}^{1/2}) _{n>N}\Vert_{l^{2}}^{2}\\
&&+C(y) (1+T)\int_{0}^{t} \Vert Y^{(
N) }(s)-Y(s)\Vert_{H}^{2}ds.
\end{eqnarray*}

By using Gronwall's lemma we deduce
\begin{equation*}
\left\Vert Y^{\left( N\right) }\left( t\right) -Y\left( t\right) \right\Vert
_{H}^{2}\leq \left( \left\Vert \Pi _{N}y-y\right\Vert _{H}^{2}+CT\left\Vert
\left( a_{n}^{1/2}\right)_{n>N}\right\Vert _{l^{2}}^{2}\right) e^{C(y,T)t}
\begin{array}{c}
\medskip \medskip \medskip
\end{array}
\end{equation*}
and since 
\begin{equation*}
\underset{N\rightarrow \infty }{\lim }\left\Vert \Pi _{N}y-y\right\Vert
_{H}^{2}=0
\begin{array}{c}
\medskip \medskip
\end{array}
\end{equation*}
and
\begin{equation*}
\underset{N\rightarrow \infty }{\lim }\left\Vert \left( a_{n}^{1/2}\right)
_{n>N}\right\Vert _{l^{2}}^{2}=0
\begin{array}{c}
\medskip 
\end{array}
\end{equation*}
we can conclude the proof of this result.
\end{proof}

\bigskip

\begin{proposition}
Under the assumptions presented above, the probability $\mu $ defined in (\ref{invar}) is an invariant measure of the transition semigroup $\left(
P_{t}\right) _{t}$ of (\ref{ecuatie}) on $H$.
\end{proposition}

\begin{proof}
We define the measure 
\begin{eqnarray*}
\mu _{\infty }^{N}\left( dy\right) &=&\frac{dx}{2\pi }\otimes
\prod\limits_{n=1}^{N}N\left( 0,\frac{1}{n^{2}}\right) du_{n}\otimes
\prod\limits_{n>N}\delta _{0}\left( du_{n}\right) 
\begin{array}{c}
\medskip \medskip \medskip \medskip 
\end{array}
\\
&&\quad \quad \otimes \prod\limits_{n=1}^{N}N\left( 0,\frac{1}{n^{2}}
\right) dv_{n}\otimes \prod\limits_{n>N}\delta _{0}\left( dv_{n}\right) 
\begin{array}{c}
\medskip \medskip \medskip \medskip
\end{array}
\\
&&\overset{Denote}{=}\frac{dx}{2\pi }\otimes \mu ^{N}\left( d\left(
u_{n}\right) _{n=1}^{N}\right) \otimes \mu ^{N+}\left( d\left( u_{n}\right)
_{n}\right) 
\begin{array}{c}
\medskip \medskip \medskip \medskip
\end{array}
\\
&&\quad \quad \quad \quad \quad \otimes \mu ^{N}\left( d\left( v_{n}\right)
_{n=1}^{N}\right) \otimes \mu ^{N+}\left( d\left( v_{n}\right) _{n}\right),
\end{eqnarray*}
where $\delta _{0}$ is the Dirac measure on $\mathbb{R}$.

\bigskip

\textbf{Step I.}

We prove that 
\begin{equation*}
\mu _{\infty }^{N}\underset{N\rightarrow \infty }{\longrightarrow }\mu 
\end{equation*}
for the topology of weak convergence, i.e., 
$$\mu _{\infty }^{N}\varphi\underset{N\rightarrow \infty }{\longrightarrow }\mu\varphi,\quad \forall \varphi\in C_{b}(H). $$

Let $\varphi \in C_{b}\left( H\right) $ and denote by $\varphi _{N}=\varphi
\left( \Pi _{N}\right)$.

We can easily see that 
\begin{eqnarray*}
\int_{H}\varphi \left( y\right) \mu _{\infty }^{N}
&=&\int_{H_{N}}\int_{l^{2}\times l^{2}}\varphi \left( y^{N},y^{\prime
}\right) \mu ^{N}\left( dy^{N}\right) \mu ^{N+}\left( dy^{\prime }\right) 
\begin{array}{c}
\medskip \medskip \medskip
\end{array}
\\
&=&\int_{H_{N}}\varphi \left( y^{N},0,..,0,...\right) \mu ^{N}\left(
dy^{N}\right) 
\begin{array}{c}
\medskip 
\end{array}
\end{eqnarray*}
and that
\begin{equation*}
\int_{H_{N}}\varphi \left( y^{N},0,..,0,...\right) \mu ^{N}\left( dy^{N}\right)
=\int_{H}\varphi _{N}\left( y\right) \mu _{\infty }^{N}\left( dy\right)
=\int_{H}\varphi _{N}\left( y\right) \mu \left( dy\right) .
\end{equation*}

This leads to 
\begin{equation*}
\int_{H}\varphi \left( y\right) \mu _{\infty }^{N}=\int_{H}\varphi
_{N}\left( y\right) \mu \left( dy\right) .
\end{equation*}

Since 
\begin{equation*}
\underset{N\rightarrow \infty }{\lim }\Pi _{N}\left( y\right) =y
\end{equation*}
and keeping in mind that $\varphi $ is bounded continuous, we have via
Lebesgue dominated convergence theorem that
\begin{equation*}
\underset{N\rightarrow \infty }{\lim }\int_{H}\varphi \left( \Pi _{N}\left(
y\right) \right) \mu \left( dy\right) =\int_{H}\varphi \left( y\right) \mu
\left( dy\right),
\end{equation*}
and consequently
\begin{equation*}
\underset{N\rightarrow \infty }{\lim }\int_{H}\varphi \left( y\right) \mu
_{\infty }^{N}\left( dy\right) =\int_{H}\varphi \left( y\right) \mu \left(
dy\right),
\end{equation*}
i.e., 
\begin{equation*}
\mu _{\infty }^{N}\underset{N\rightarrow \infty }{\longrightarrow }\mu.
\end{equation*}

\bigskip

\textbf{Step II}.

We show that $\mu $ is an invariant measure for the transition semigroup.

Let $P_{t}^{N}$ be the transition semigroup corresponding to (\ref{aprox}).
We take
\begin{eqnarray}
\int_{H}P_{t}\varphi \left( y\right) \mu \left( dy\right) &=&\int_{H}\left(
P_{t}\varphi \left( y\right) -P_{t}^{N}\varphi \left( \Pi _{N}y\right)
\right) \mu \left( dy\right) 
\begin{array}{c}
\medskip \medskip \medskip
\end{array}
\label{diff} \\
&&+\int_{H}\left( P_{t}^{N}\varphi \left( \Pi _{N}y\right) \right) \mu
\left( dy\right) 
\begin{array}{c}
\medskip 
\end{array}
\notag \\
&&\overset{Denote}{=}\varepsilon _{N}+\int_{H}\left( P_{t}^{N}\varphi \left(
\Pi _{N}y\right) \right) \mu \left( dy\right) .
\begin{array}{c}
\medskip \medskip \medskip
\end{array}
\notag
\end{eqnarray}

By the same arguments developed in \cite{CEGMB} one can prove that $\mu ^{N}$ is an invariant measure for $P_{t}^{N}$. So we obtain that
\begin{eqnarray*}
\int_{H}\left( P_{t}^{N}\varphi \left( \Pi _{N}y\right) \right) \mu \left(
dy\right) &=&\int_{H_{N}}P_{t}^{N}\varphi \left( y^{N},0,..,0,...\right) \mu
^{N}\left( dy^{N}\right) 
\begin{array}{c}
\medskip 
\end{array}
\\
&=&\int_{H_{N}}\varphi \left( y^{N},0,..,0,...\right) \mu ^{N}\left( dy^{N}\right) 
\begin{array}{c}
\medskip
\end{array}
\\
&=&\int_{H}\varphi \left( y\right) \mu _{\infty }^{N}\left( dy\right).
\end{eqnarray*}

On the other hand we have that 
\begin{equation*}
P_{t}\varphi \left( y\right) -P_{t}^{N}\varphi \left( \Pi _{N}y\right)
\medskip \medskip =\mathbb{E}\left( \varphi \left( Y_{t}\right) -\varphi
_{N}\left( Y_{t}^{\left( N\right) }\right) \right) .
\end{equation*}

By Lemma \ref{CVGE}, we have
\begin{equation*}
Y_{t}^{\left( N\right) }\rightarrow Y_{t}\text{ a.s. }
\end{equation*}
for $N\rightarrow \infty ,$ and thus for $\varphi \in C_{b}\left( H\right)$, we obtain that
\begin{equation*}
\varphi(Y_{t}^{\left( N\right) })\rightarrow \varphi (Y_{t}) \text{ a.s. }.
\end{equation*}
  Because $\varphi (Y_{t}^{\left( N\right) })$ is bounded, this leads to
\begin{equation*}
P_{t}^{N}\varphi \left( \Pi _{N}y\right) \rightarrow P_{t}\varphi \left(
y\right) ,
\end{equation*}
for $N\rightarrow \infty$ by the Dominated Convergence Theorem. 

Since $Y_{t}^{\left( N\right) }$ is a Feller process, we have that $
P_{t}^{\left( N\right) }\varphi \in C_{b}\left( H\right) $ for all $\varphi
\in C_{b}\left( H\right) $, and we get via the Lebesgue Dominated Convergence Theorem
\begin{equation*}
\varepsilon _{N}=\int_{H}\left( P_{t}\varphi \left( y\right)
-P_{t}^{N}\varphi \left( \Pi _{N}y\right) \right) \mu \left( dy\right)
\rightarrow 0,
\end{equation*}
for $N\rightarrow \infty .$

Going back to (\ref{diff}) and passing to the limit for $N\rightarrow \infty 
$ we get that
\begin{equation*}
\int_{H}P_{t}\varphi \left( y\right) \mu \left( dy\right) =\underset{N\rightarrow \infty }{\lim }\int_{H}\varphi \left( y\right) \mu _{\infty
}^{N}\left( dy\right) =\int_{H}\varphi \left( y\right) \mu \left( dy\right) .
\begin{array}{c}
\medskip 
\end{array}
\end{equation*}

The existence of an invariant measure is now completely proved.
\end{proof}

\section{ On the uniqueness of the invariant measure }

In this section, we intend to give an important feature for the Kolmogorov operator $L$. Keeping in mind that $\sigma$ is the projection on the first coordinate, we have
\begin{eqnarray}\label{KolmogOp}
L\varphi (y) = \frac{1}{2}\partial _{xx}\varphi \left(
y\right) +\partial _{x}\varphi \left( y\right) \sum_{n}na_{n}^{1/2}\left(
v_{n}\cos \left( nx\right) +u_{n}\sin \left( nx\right) \right)\\
+\sum_{n}a_{n}^{1/2}\cos \left( nx\right) \partial
_{u_{n}}\varphi \left( y\right)
-\sum_{n}a_{n}^{1/2}\sin \left( nx\right) \partial
_{v_{n}}\varphi \left( y\right)\notag
\end{eqnarray}
for $\varphi\in C_{b}^{2}(H)$ the class of all bounded functions which are twice Fr\'echet differentiable and whose derivatives are bounded. 
 
We recall that the Kolmogorov operator associated to (\ref{ecuatie}) is obtained by using It\^{o} formula to function $\varphi$ in $C_{b}^{2}(H)$ (for details see Theorem 5.4.2 from page 72 of \cite{DPZ2}).

Set $S\varphi(y)=\frac{1}{2}\partial_{xx}\varphi(y)$ and $A\varphi(y):=L\varphi(y)-S\varphi(y)$.
\begin{lemma}\label{symetry}
For two functions $\psi$ and $\varphi$ in $C_{b}^{2}(H)$, we have
\begin{equation}\label{Sope}
\int_{H}S\varphi(y)\psi(y)\mu (dy)=\int_{H}\varphi(y)S\psi(y)\mu(dy)=-\frac{1}{2}\int_{H}\partial_{x}\varphi (y)\partial_{x}\psi (y)\mu (dy)
\end{equation} 
and 
\begin{equation}\label{Aope}
\int_{H}A\varphi(y)\psi(y)\mu (dy)=-\int_{H}\varphi(y)A\psi(y)\mu(dy)
\end{equation}
\end{lemma}
\begin{proof} Let $\varphi \in C_{b}^{2}(H)$. It is trivial that $\varphi,\, S\varphi\in L^{2}(H,\mu)$ by definition of $C_{b}^{2}(H)$. We shall start by proving that $A\varphi\in L^{2}(H,\mu)$.

From the definition of $A$, we have
\begin{eqnarray*}
A\varphi \left( y\right) &=&\partial _{x}\varphi (y) (\langle (na_{n}^{1/2}\cos(nx))_{n\geqslant 1} , v\rangle_{l^{2}} +\langle (na_{n}^{1/2}\sin(nx))_{n\geqslant 1} , u\rangle_{l^{2}} ) \\
\begin{array}{c}
\medskip \medskip \medskip
\end{array}
& & +\langle (a_{n}^{1/2}\cos(nx))_{n\geqslant 1} , \nabla_{u}\varphi (y)\rangle_{l^{2}} +\langle (-a_{n}^{1/2}\sin(nx))_{n\geqslant 1} , \nabla_{v}\varphi (y)\rangle_{l^{2}}
\end{eqnarray*}

Therefore, by using the inequality $(\sum_{j=1}^{n}x_{j})^{2}\leqslant n\sum_{j=1}^{n}x_{j}^{2}$ with $n=4$, we obtain
\begin{eqnarray*}
A\varphi \left( y\right)^{2} &\leqslant& 4\partial _{x}\varphi (y)^{2} (\langle (na_{n}^{1/2}\cos(nx))_{n\geqslant 1} , v\rangle_{l^{2}}^{2} +\langle (na_{n}^{1/2}\sin(nx))_{n\geqslant 1} , u\rangle_{l^{2}}^{2} ) \\
\begin{array}{c}
\medskip \medskip \medskip
\end{array}
& & +4\langle (a_{n}^{1/2}\cos(nx))_{n\geqslant 1} , \nabla_{u}\varphi (y)\rangle_{l^{2}}^{2} +4\langle (-a_{n}^{1/2}\sin(nx))_{n\geqslant 1} , \nabla_{v}\varphi (y)\rangle_{l^{2}}^{2}\\
\begin{array}{c}
\medskip 
\end{array}
&\leqslant& C(1+\Vert u\Vert_{l^{2}}^{2}+\Vert v\Vert_{l^{2}}^{2})
\end{eqnarray*}
where the last inequality is obtained by the Cauchy-Schwarz inequality and $C$ is a constant depending on $(a_{n})_{n}$ and on the upper bounds of the derivatives of $\varphi$. Hence $A\varphi\in L^{2}(H,\mu)$.

Since $\mu $ has $\mathbb{S}^{1}\times l^{2}\times l^{2}$ as support, we may
extend $\varphi $ to $\widetilde{H}=\mathbb{S}^{1}\times \mathbb{R}^{\mathbb{\infty }%
}\times \mathbb{R}^{\mathbb{\infty }}$ by the same expression.

Therefore 
\begin{eqnarray*}
\int_{H}S\varphi(y)\psi(y)\mu (dy)&=&\int_{\widetilde{H}}S\varphi(y)\psi(y)\mu (dy)\\
&=&\int_{\widetilde{H}}\frac{1}{2}\partial _{xx}\varphi \left( y\right)
\psi \left( y\right) \mu \left( dy\right) 
\begin{array}{c}
\medskip \medskip \medskip
\end{array}
\label{carl} \\
&=&\int_{\mathbb{R}^{\infty }\times \mathbb{R}^{\infty }}\int_{\mathbb{S}^{1}}\frac{1%
}{4\pi}\partial _{xx}\varphi \left( y\right) \psi \left( y\right) dx~N\left(
0,Q\right) \left( d\left( u_{n}\right) _{n}\right) ~N\left( 0,Q\right)
\left( d\left( v_{n}\right) _{n}\right) 
\begin{array}{c}
\medskip \medskip \medskip
\end{array}
\notag \\
&=&-\frac{1}{2}\int_{H}\partial_{x}\psi (y)\partial_{x}\varphi (y)\mu (dy),%
\begin{array}{c}
\medskip \medskip \medskip
\end{array}
\notag
\end{eqnarray*}
where $N\left( 0,Q\right) \left( d\left( u_{n}\right) _{n}\right)
=\prod\limits_{n\geq 1}N\left( 0,\frac{1}{n^{2}}\right) du_{n}$ and
similarly for $N\left( 0,Q\right) \left( d\left( v_{n}\right) _{n}\right) .$ This proves (\ref{Sope}). 

Furthermore
\begin{eqnarray*}
&&\int_{H}A\varphi \left( y\right) \psi \left( y\right) \mu \left(
dy\right) 
\begin{array}{c}
\medskip \medskip \medskip 
\end{array}
\\
&=&\int_{\widetilde{H}}A\varphi \left( y\right) \psi \left( y\right) \mu
\left( dy\right) 
\begin{array}{c}
\medskip \medskip \medskip
\end{array}
\\
&=&\int_{\widetilde{H}}\left( \partial _{x}\varphi \left( y\right) \sum_{n}na_{n}^{1/2}\left(
v_{n}\cos \left( nx\right) +u_{n}\sin \left( nx\right) \right) \right)
\psi \left( y\right) \mu \left( dy\right) 
\begin{array}{c}
\medskip \medskip \medskip 
\end{array}
\\
&&+\int_{\widetilde{H}}\sum_{n}a_{n}^{1/2}\cos \left( nx\right) \partial
_{u_{n}}\varphi \left( y\right) \psi \left( y\right) \mu \left( dy\right) 
\begin{array}{c}
\medskip \medskip \medskip
\end{array}
\\
&&-\int_{\widetilde{H}}\sum_{n}a_{n}^{1/2}\sin \left( nx\right) \partial
_{v_{n}}\varphi \left( y\right) \psi \left( y\right) \mu \left( dy\right).
\begin{array}{c}
\medskip \medskip \medskip
\end{array}
\end{eqnarray*}

For the term
\begin{equation*}
\int_{\widetilde{H}}\sum_{n}na_{n}^{1/2}v_{n}\cos \left( nx\right) \partial
_{x}\varphi \left( y\right) \psi \left( y\right) \mu \left( dy\right)
\end{equation*}
we compute 
\begin{equation*}
\int_{\mathbb{S}^{1}}\cos \left( nx\right) \psi \left( y\right) \partial
_{x}\varphi \left( y\right) dx=-\int_{\mathbb{S}^{1}}\partial _{x}\left( \psi
\left( y\right) \cos \left( nx\right) \right) \varphi \left( y\right) dx
\end{equation*}
\begin{equation}
=-\int_{\mathbb{S}^{1}}\partial _{x}\left( \psi \left( y\right) \right) \cos
\left( nx\right) \varphi \left( y\right) dx+\int_{\mathbb{S}^{1}}n\sin \left(
nx\right) \psi\left( y\right)\varphi \left( y\right) dx  \label{carl2}
\end{equation}
and then we get 
\begin{eqnarray*}
& &\int_{\widetilde{H}}\sum_{n}na_{n}^{1/2}v_{n}\cos \left( nx\right)
\partial _{x}\varphi \left( y\right) \psi \left( y\right) \mu \left(
dy\right)\\
 &=&\int_{\widetilde{H}}\sum_{n}n^{2}a_{n}^{1/2}v_{n}\sin \left( nx\right)
\varphi \left( y\right) \psi \left( y\right) \mu \left(
dy\right)\\
& & -\int_{\widetilde{H}}\sum_{n}na_{n}^{1/2}v_{n}\cos \left( nx\right)
\partial _{x}\psi \left( y\right) \varphi \left( y\right) \mu \left(
dy\right)
\end{eqnarray*}

Moreover for the term 
\begin{equation*}
-\int_{\widetilde{H}}\sum_{n}a_{n}^{1/2}\sin \left( nx\right) \partial
_{v_{n}}\varphi \left( y\right) \psi \left( y\right) \mu \left( dy\right)
\end{equation*}
we have
\begin{eqnarray}
&&\int_{\mathbb{R}}\partial _{v_{n}}\varphi \left( y\right) \psi \left(
y\right) e^{-\frac{n^{2}}{2} v_{n} ^{2}}dv_{n}%
\begin{array}{c}
\medskip \medskip \medskip
\end{array}
\label{carl3} \\
&=&-\int_{\mathbb{R}}\varphi \left( y\right) \partial _{v_{n}}\left( \psi
\left( y\right) e^{-\frac{n^{2}}{2} v_{n}^{2}}\right) dv_{n}
\begin{array}{c}
\medskip \medskip \medskip
\end{array}
\notag \\
&=&-\int_{\mathbb{R}}\varphi \left( y\right) \partial _{v_{n}}\left( \psi
\left( y\right) \right) e^{-\frac{n^{2}}{2} v_{n} ^{2}}dv_{n}
\begin{array}{c}
\medskip \medskip \medskip
\end{array}
\notag \\
&&\quad \quad \quad \quad +\int_{\mathbb{R}}\varphi \left( y\right)\psi \left( y\right)
n^{2}v_{n}e^{-\frac{n^{2}}{2} v_{n} ^{2}}dv_{n},
\begin{array}{c}
\medskip \medskip \medskip
\end{array}
\notag
\end{eqnarray}
which yields 
\begin{eqnarray*}
&&-\int_{\widetilde{H}}\sum_{n}a_{n}^{1/2}\sin \left( nx\right) \partial
_{v_{n}}\varphi \left( y\right) \psi \left( y\right) \mu \left( dy\right) 
\begin{array}{c}
\medskip \medskip \medskip
\end{array}
\\
&=& \int_{\widetilde{H}}\sum_{n}a_{n}^{1/2}\sin \left( nx\right) \partial
_{v_{n}}\psi \left( y\right) \varphi \left( y\right) \mu \left( dy\right)
\begin{array}{c}
\medskip \medskip \medskip
\end{array}
\\
&& - \int_{\widetilde{H}}\sum_{n}n^{2}a_{n}^{1/2}v_{n}\sin \left( nx\right)
\psi \left( y\right) \varphi \left( y\right) \mu \left(
dy\right)
\end{eqnarray*}

Similarly to (\ref{carl2}) and (\ref{carl3}) we get 
\begin{eqnarray}
&&\int_{\widetilde{H}}\sum_{n}na_{n}^{1/2}u_{n}\sin \left( nx\right)
\partial _{x}\varphi \left( y\right) \psi \left( y\right) \mu \left(
dy\right) 
\begin{array}{c}
\medskip \medskip \medskip
\end{array}
\label{carl4} \notag\\
&=& -\int_{\widetilde{H}}\sum_{n}n^{2}a_{n}^{1/2}u_{n}\cos \left( nx\right)
\varphi \left( y\right) \psi \left( y\right) \mu \left(
dy\right)\\
& & -\int_{\widetilde{H}}\sum_{n}n a_{n}^{1/2}u_{n}\sin \left( nx\right)
\partial _{x}\psi \left( y\right) \varphi \left( y\right) \mu \left(
dy\right)\notag
\end{eqnarray}
and 
\begin{eqnarray}
&&\int_{\widetilde{H}}\sum_{n}a_{n}^{1/2}\cos \left( nx\right) \partial
_{u_{n}}\varphi \left( y\right) \psi \left( y\right) \mu \left( dy\right) 
\begin{array}{c}
\medskip \medskip \medskip
\end{array}
\label{carl5}\notag \\
&=& \int_{\widetilde{H}}\sum_{n}n^{2}a_{n}^{1/2}u_{n}\cos \left( nx\right) 
\varphi \left( y\right) \psi \left( y\right) \mu \left( dy\right)
\begin{array}{c}
\medskip \medskip \medskip
\end{array}
\\
&&- \int_{\widetilde{H}}\sum_{n}a_{n}^{1/2}\cos \left( nx\right) \partial
_{u_{n}}\psi \left( y\right) \varphi \left( y\right) \mu \left( dy\right)\notag
\end{eqnarray}

Putting (\ref{carl}) to (\ref{carl5}) altogether gives (\ref{Aope}).
\end{proof}

An easy consequence of the result above is the following.
\begin{corollary}\label{c1}
For a function $\varphi \in C_{b}^{2}(H) $,
we have 

\begin{equation*}
\int_{H}L\varphi \left( y\right) \varphi \left( y\right) \mu \left(
dy\right) =-\frac{1}{2}\int_{H}\left\vert \partial _{x}\varphi \left(
y\right) \right\vert ^{2}\mu \left( dy\right) .
\end{equation*}
Furthermore, if $\varphi$ is such that $L\varphi =0$, then $\varphi $ is
constant on $H$. 
\end{corollary}

\begin{proof}
Let $\varphi \in C_{b}^{2}(H)$. By Lemma \ref{symetry}, we have
$$\int_{H}A\varphi(y)\varphi(y)\mu (dy)=-\int_{H}\varphi(y)A\varphi(y)\mu(dy); $$
hence $\int_{H}A\varphi(y)\varphi(y)\mu (dy)=0.$

Thus
\begin{eqnarray}\label{Lff}
\int_{H}L\varphi \left( y\right) \varphi \left( y\right) \mu \left(
dy\right) &=& \int_{H}S\varphi \left( y\right) \varphi \left( y\right) \mu \left(
dy\right)\notag\\
&=&-\frac{1}{2}\int_{H}\left\vert \partial _{x}\varphi \left(
y\right) \right\vert ^{2}\mu \left( dy\right).
\end{eqnarray}


Assume now that $\varphi$ satisfies $L\varphi =0$. Then, by (\ref{Lff}), we obtain
\begin{equation*}
0=-\frac{1}{2}\int_{H}\left\vert \partial _{x}\varphi \left(
y\right) \right\vert ^{2}\mu \left( dy\right) .
\begin{array}{c}
\medskip \medskip \medskip
\end{array}
\end{equation*}
Since $\mu $ has full support on $H$ and $\partial _{x}\varphi $ is
continuous, it follows that 
\begin{equation*}
\partial _{x}\varphi \equiv 0,
\end{equation*}
i.e., $\varphi $ is independent of the $x$ variable on $H.$

Therefore
\begin{eqnarray}
0 &=&L\varphi \left( x,\left( u_{n}\right) _{n},\left( v_{n}\right)
_{n}\right) 
\begin{array}{c}
\medskip \medskip 
\end{array}
\label{carl6} \\
&=&\sum_{n}a_{n}^{1/2}\cos \left( nx\right) \partial _{u_{n}}\varphi \left(
y\right) -\sum_{n}a_{n}^{1/2}\sin \left( nx\right) \partial _{v_{n}}\varphi
\left( y\right) 
\begin{array}{c}
\medskip \medskip 
\end{array}
\notag
\end{eqnarray}
for all $\left( x,\left( u_{n}\right) _{n},\left( v_{n}\right) _{n}\right)
\in H$.

Since $\left\{ \left( \cos nx\right) _{n},\left( \sin nx\right) _{n}\right\} 
$ forms an orthogonal basis of $L^{2}\left( \mathbb{S}^{1},dx\right)$, the relation (\ref{carl6}) forces to have 
\begin{equation*}
\partial _{u_{n}}\varphi =0=\partial _{v_{n}}\varphi ,\quad \text{for all }
n\geq 1
\end{equation*}
on $H,$\ because $a_{n}$ is supposed to be strictly positive.
Consequently, $\varphi $ is a constant on $H.$
\end{proof}

Set $L^{\ast}=S-A$. Then, by applying Lemma \ref{symetry}, one can check that 
\begin{equation*}
\int_{H}L\varphi(y)\psi(y)\mu (dy)=\int_{H}\varphi(y)L^{\ast}\psi(y)\mu(dy)
\end{equation*}
for all $\varphi,\psi\in C_{b}^{2}(H)$.

Let $\nu$ be any invariant probability measure of (\ref{sist}). 

We shall explain why we believe that $\nu$ should be identical to the measure $\mu$ defined in (\ref{invar}), which would prove uniqueness of the invariant probability, as well as the ergodicity of $\mu$.\\
By the Lebesgue's decomposition theorem, there exists a positive function $g\in L^{1}(H,\mu)$ and a measure $\nu_{s}$ which is singular to $\mu$, such that 
\begin{equation*}
\nu =g\mu + \nu_{s}.
\end{equation*}
Since $\mu$ and $\nu$ are both invariant for (\ref{sist}), it follows that $g\mu$ and $\nu_{s}$ are also invariant. We can now formulate the following result.
\begin{proposition} Assume that the function $g$ defined above lies in $C_{b}^{4}(H)$. Then $g$ is constant.
\end{proposition}
\begin{proof} Since $g\mu$ is invariant, we obtain that 
\begin{eqnarray}
0=\int_{H}Lg(y) (g\mu)(dy)&=&\int_{H}Lg(y)g(y)\mu(dy) \label{star}\\
&=&-\frac{1}{2}\int_{H}\left\vert \partial _{x}g \left(
y\right) \right\vert ^{2}\mu \left( dy\right) \notag
\end{eqnarray}
by Corollary \ref{c1}.

Therefore we deduce that $\partial_{x}g\equiv 0$ by the continuity of $\partial_{x}g$ and the full support of $\mu$. Hence
\begin{eqnarray*}
Lg(y)&=&\sum_{n}a_{n}^{1/2}\cos \left( nx\right) \partial_{u_{n}}g \left( y\right)-\sum_{n}a_{n}^{1/2}\sin \left( nx\right) \partial_{v_{n}}g \left( y\right)\\
&=&-L^{\ast}g(y).
\end{eqnarray*}
By the Cauchy-Schwarz inequality and the definition of the space $C_{b}^{4}(H)$, it is clear that $Lg$ is bounded and consequently that $Lg\in C_{b}^{2}(H)$ as well as $L^{\ast}g$. Therefore, by application of (\ref{star}) with $L^{\ast}g$ in place of $g$, we get
\begin{eqnarray*}
0=\int_{H}L(L^{\ast}g)(y) (g\mu)(dy)&=&\int_{H}L(L^{\ast}g)(y)g(y)\mu(dy)\\
&=&\int_{H}(L^{\ast}g)^{2}(y)\mu(dy) ,
\end{eqnarray*}
which leads to $L^{\ast}g\equiv 0$ and so does $Lg$. By Corollary \ref{c1}, we get that $g$ is constant.
\end{proof}
 
A straightforward consequence is the following result.

\begin{corollary} If $\nu$ is absolutely continuous with respect to $\mu $ and such that its Radon-Nikodym derivative lies in $C_{b}^{4}(H)$, then $\nu=\mu$.
\end{corollary}
\begin{remark}
\begin{enumerate}
\item The proposition still holds true for $g\in C_{b}^{2}(H)$ since $L$ is well defined on $C_{b}^{2,1,1}(H)$, the set of bounded functions which are twice differentiable in $x$, and once differentiable in $u$ and $v$ and such that these partial derivatives are bounded.
\item If the function $g$ in the proposition has bounded support then $g\equiv 0$ and so $\nu$ is singular to $\mu$.
\end{enumerate}
\end{remark}

\section{Conclusion}

In this work, we aim to generalize the setting of \cite{CEGMB} to the infinite dimensional case, at least for the case of the unit circle. Since our non-linear operator $F$ is neither Lipschitz nor monotone we could not directly apply classic results in the sense that we had to prove some additional properties which hold for $F$.

As mentioned at the beginning of the Section \ref{s5}, we succeed to prove that a natural generalization of the invariant measure in the finite dimensional case was indeed an invariant measure in our setting.

However, we were not yet able to obtain its uniqueness, while in \cite{CEGMB} it is the case. This is due to the fact that we could not use H\"{o}rmander's like condition to get the strong Feller property which was the main argument in the finite dimensional case. So at this point, a first question is

\begin{flushleft}
1. Do we have uniqueness for the invariant measure?
\end{flushleft}

Thanks to Corollary \ref{c1}, we think that it might be the case.

If this is not true, a second open question would be 

\begin{flushleft}
2. Is $\mu$ an ergodic measure, which means that, if $A\in \mathcal{B}(H)$ is such that $P_{t}1_{A}=1_{A}$, then $\mu(A)\in\{ 0,1\}$?
\end{flushleft}

As mentioned above, the strong Feller property was proved in \cite{CEGMB}, while in our case, it does not hold (see remark \ref{noStrongFell}). On the other hand, the question of having asymptotically strong Feller property is still open (see paragraph 11 in \cite{hairer} for the definition). More precisely, in order to ensure that all our computations make sense, we had to choose our coefficients $(a_{n})_{n}$ in $O^{5}$; so the question can be formulated as 

\begin{flushleft}
3. If $(a_{n})_{n}\in \cap_{k\geqslant 1} O^{k}$ for example, do we have the asymptotic strong Feller property? If yes, can we weaken the assumption on the sequence $(a_{n})_{n}$?
\end{flushleft}

Finally, in the case of positive answer to this last question, the answer for the first will be positive since $\mu$ has full support.     

\section*{Acknowledgements}
The authors thanks the anonymous referee for his careful reading and its constructive comments and suggestions that improved the presentation of the paper.

\newpage

\textbf{Michel Bena\"im}, e-mail: \textit{michel.benaim(AT)unine.ch}\\
Institut de Math\'ematiques, Universit\'e de Neuch\^atel,\\
Rue \'Emile Argand 11, 2000 Neuch\^atel, Switzerland.\\

\textbf{Ioana Ciotir}, e-mail: \textit{ioana.ciotir(AT)unine.ch}\\
Institut de Math\'ematiques, Universit\'e de Neuch\^atel,\\
Rue \'Emile Argand 11, 2000 Neuch\^atel, Switzerland.\\

\underline{Present address:} Laboratoire de Math\'ematiques de l'INSA de Rouen, Avenue de l'Universit\'e, 76800 St Etienne du Rouvray, France.  \\
email: \textit{ioana.ciotir(AT)insa-rouen.fr} \\

\textbf{Carl-Erik Gauthier}, e-mail: \textit{carl-erik.gauthier(AT)unine.ch}\\
Institut de Math\'ematiques, Universit\'e de Neuch\^atel,\\
Rue \'Emile Argand 11, 2000 Neuch\^atel, Switzerland.\\
\end{document}